\documentclass[11pt]{amsart}
\usepackage{paper}
\usepackage{float}
\usepackage{longtable}
\renewcommand{\C}{\mathcal{C}}

\definecolor{mygreen}{rgb}{0.35, 0.71, 0.0}

\newcommand{\ordpi}{\ord_{p,i}}

\title{Connectedness of special points in the Markoff $\mod p$ graphs}

\author{Elisa Bellah}
\address{(Bellah) Department of Mathematics, University of Toronto}
\email{elisa.bellah@utoronto.ca}

\author{Claire Dunn}
\address{(Dunn) Department of Mathematics, Oregon State University}
\email{dunncla@oregonstate.edu}

\author{Vernon Naidu}
\address{(Naidu) Department of Mathematics, San Francisco State University}
\email{vnaiduteaching@gmail.com}

\author{Alette Wells}
\address{(Wells) Department of Mathematics, The University of Chicago}
\email{alettew@uchicago.edu}

\begin{document}

\begin{abstract} It is conjectured that the Markoff equation $X^2+Y^2+Z^2=3XYZ$ satisfies the special Diophantine property that every $\mod p$ solution lifts to an integer solution. Progress toward this conjecture has been made by studying the connectedness of the graphs obtained from the action of the Vieta group on the nonzero $\mod p$ solutions to the Markoff equation. In this paper, we use results on Pisano periods of the Fibonacci sequence to obtain explicit results on the connectedness of special points in this graph for primes $p$ where $p+1$ has large $2$-adic valuation. In particular, for Mersenne primes $p \equiv \pm 2 \,(\mod 5)$, we show that the special point $(1, 1, 1)$ which is fixed under reduction modulo $p$ lies in a component of this graph which is known to be connected.
\end{abstract}

\maketitle

\section{Introduction}

The Markoff surface $\X$ is the affine surface in $\A^3$ given by
\[\X: X_1^2+X_2^2+X_3^2=3X_1X_2X_3,\]
and the positive integer points $\X(\Z_{>0})$ are called Markoff triples. Originating in the theory of Diophantine approximation and quadratic forms (see \cite{markoff}), this equation has more recently garnered interest as it appears to satisfy the special Diophantine property that every $\mod p$ solution has an integer lift, as conjectured by Baragar in \cite{baragar}. Major progress toward this conjecture (often referred to as ``Strong Approximation") was made in the paper \cite{BGS} of Bourgain, Gamburd, and Sarnak by studying the connectedness of the graphs arising from the action of the Vieta group (discussed in Section \ref{sec:2}) on the $\mod p$ solutions.\\

In this paper, we build on the results of \cite{BGS} and \cite{bellah} to construct an explicit family of primes where Strong Approximation holds for an expected ``large proportion" of $\mod p$ solutions to the Markoff equation. Below we state a consequence of our main result, which builds on the major progress from \cite{BGS}.

\begin{theorem}[Consequence of Theorem \ref{main}] \label{mainprelim}
Let $p>5$ be a Mersenne prime with $p \equiv \pm 2 (\mod 5)$.
If $\x=(x_1, x_2, x_3)$ is a $\mod p$ solution to the Markoff equation where the order of \[\begin{pmatrix} 0 & 1 \\ -1 & 3x_i \end{pmatrix}\] in $\GL_2(\F_p)$
is at least $p+1$ for some $i \in \{1, 2, 3\}$, then $\x$ has an integer lift.
\end{theorem}

Our contribution to Theorem \ref{mainprelim} guarantees the connectedness of the special point $(1, 1, 1)$ which is fixed under reduction modulo $p$ to a large component of the Markoff $\mod p$ graph called \textit{the cage} (defined in Section \ref{sec:2}). Our main result (Theorem \ref{main}) will demonstrate a family of primes $p$ and special points in the Markoff $\mod p$ graphs which are guaranteed to be connected to the cage, yielding Theorem \ref{mainprelim} as a corollary. As discussed in Section 3.5 of \cite{bellah}, our result can also be used to relax the conditions of Theorem 1.3 from \cite{bellah}, which upper bounds the size of minimal lifts of $\mod p$ solutions to the Markoff equation. \\

This paper is organized as follows. In Section \ref{sec:2} we give the background needed to state our main result (Theorem \ref{main}). In particular, we define the Markoff $\mod p$ graphs and discuss how Theorem \ref{mainprelim} is related to the connectedness of the special point $(1, 1, 1)$ to the cage (see Remark \ref{why111}). In Section \ref{sec:3} we show that the rotation order of the special point (1, 1, 1), which is key for the analysis of the Markoff $\mod p$ graphs outlined in \cite{BGS}, is equal to half of the Pisano period of the Fibonacci sequence. In Section \ref{sec:4}, we generalize a result on Pisano periods given in \cite{vince} to prove our main result. In Section \ref{sec:5} we discuss the density of primes where the special point $(1, 1, 1)$ is the cage.

\subsection*{Acknowledgements} 
This project was funded by Carnegie Mellon University's SUAMI program. We thank the SUAMI organizers for supporting this project. We also thank Professor Elena Fuchs for initiating work that inspired this project. 

\section{Preliminaries} \label{sec:2}
In this section, we give the background needed to state our main result (Theorem \ref{main}). This section largely follows Section 1 of \cite{bellah}. We omit the proofs of results in this section, and instead refer the reader to \cite{bellah} and \cite{BGS}. \\

For convenience, we denote the nonzero $\mod p$ solutions to the Markoff equation by $\X^*(p)$. The Vieta group $\Gamma$ is the group of affine integral morphisms on $\A^3$ generated by permutations $\sigma_{ij}$ and Vieta involutions $R_i$ (see \cite{BGS} or \cite{bellah} for details). We will consider the followings special elements of the Vieta group introduced in \cite{BGS}. 

\begin{definition} The \textit{rotations} are the elements $\rot_i$ of $\Gamma$ given by
\[\rot_1(x_1, x_2, x_3) = (x_1, x_3, 3x_1x_3-x_2)\]
\[\rot_2(x_1, x_2, x_3) = (x_3, x_2, 3x_2x_3 - x_1)\]
\[\rot_3(x_1, x_2, x_3) = (x_2, 3x_2x_3-x_1, x_3).\]

For a prime $p$, the \textit{Markoff $\mod p$ graph} $\G_p$ is defined to be the graph with vertex set $\X^*(p)$ and edges $(\x, \rot_i \x)$ for $i \in \{1, 2, 3\}$. 
	\end{definition}

In \cite{BGS}, the authors construct paths in $\G_p$ by analyzing the orbits under each of the three rotations. We give some notation and terminology to these orbits and their lengths.

\begin{definition} For $\x = (x_1, x_2, x_3) \in \X^*(p)$, we define the following.
	\begin{enumerate}
		\item The \textit{$i$th rotation order} of $\x$ is given by
		\[\ord_{p, i}(\x) := \min \{n \in \Z_{>0} \mid \rot_i^n (\x) \equiv \x (\mod p)\}.\]
		\item The \textit{rotation order of $\x$} is given by
		\[ \ord_p(\x) : = \max\{ \ord_{p, i}(x) \mid i=1, 2, 3\}.\]
	\end{enumerate}
\end{definition}

Observe that for distinct $i, j, k$, the rotation $\rot_i$ acts on $(x_j, x_k)$, and we \\can write
$$\rot_i(x_i) \begin{pmatrix} x_j \\ x_k \end{pmatrix} = \begin{pmatrix} 0  & 1 \\ -1 & 3x_i \end{pmatrix} \begin{pmatrix} x_k \\ x_j \end{pmatrix}.$$
So, it can be shown that the $i$th rotation order of $\x$ is
equal to the order of
\[A_{\x, i}:=\begin{pmatrix} 0 & 1 \\ -1 & 3x_i \end{pmatrix}\]
in $\GL_2(\F_p)$. 
Note in particular that $\rot_{p, i}(\x)$ only depends on the $i$th coordinate of $\x$. \\

	We set the following notation. For $\x=(x_1, x_2, x_3) \in \X^*(p)$, we denote
 the characteristic polynomial of $A_{\x, i}$ by $f_{\x, i}$ and the discriminant of $f_{\x, i}$ by $\Delta_{\x, i}$. When we are only concerned with a single coordinate $x$ of a Markoff $\mod p$ point, we will instead use the notation  
$A_{x}, f_x$ and $\Delta_{x}$, respectively. \\

Note that the action of $\rot_i$ on a Markoff triple $\x$ leaves the $i$th coordinate of $\x$ fixed, and so the orbits $\{\rot_i^n (\x) \mid n \in \Z\}$ correspond to points inside of some conic section with discriminant $\Delta_{\x, i}$. 
Accordingly, we have the following definition, as given in \cite{BGS}. 

\begin{definition}\label{sectionsdef} For $x \in \F_p$ we say that $x$ is
	\begin{enumerate}
	\item \textit{parabolic} if $\Delta_{x} \equiv 0 (\mod p)$,  
	\item \textit{hyperbolic} if $\Delta_x$ is a nonzero square modulo $p$, and
	\item \textit{elliptic} if $\Delta_x$ is not a square modulo $p$.
	\end{enumerate} 
\end{definition}

We have the following observations, whose proofs can be found in Lemma 2.9 and Proposition 2.10 of \cite{bellah}.

\begin{lemma} \label{ordepsilon} Let $\x=(x_1, x_2, x_3)$ be a Markoff $\mod p$ point, and set $x=x_i$. Let $\epsilon_x$ a root of $f_x$. If $x$ is not parabolic, then $\ord_{p, i}(\x)$ is equal to the order of $\epsilon_x$ in $\F^\times$, where 
	\[\F= \begin{cases} \F_p & \text{ when $x$ is hyperbolic} \\
		\F_{p^2}& \text{ when $x$ is elliptic} \end{cases}\] 
under the identification $\F_{p^2} \cong \F_p[T]/(T^2-\Delta_x).$
\end{lemma}

\begin{proposition} \label{orders_upperbound}
Let $\x=(x_1, x_2, x_3) \in \X^*(p)$. For a prime $p >3$. We have 
	\[\ord_{p, i}(\x) \text{ divides }
	\begin{cases} p-1 & \text{ if $x_i$ is hyperbolic} 
		\\ p+1 & \text{ if $x_i$ is elliptic}.
	\end{cases}
	\]
If $x_i$ is parabolic, then $x_i=\pm 2/3$ and we have
\[\ord_{p, i}(\x) = \begin{cases} 2p & \text{ if } x_i = -2/3 \\ p & \text{ if } x_i=2/3. \end{cases}\]

\end{proposition}

Next, we define the following, as in \cite{BGS}.

\begin{definition} \label{maxord} Let $\x=(x_1, x_2, x_3) \in \X^*(p)$. 
	\begin{enumerate}
		\item If $\ord_{p, i}(\x)=p-1$ then we say $x_i$ is \textit{maximal hyperbolic}. 
		\item If $\ord_{p, i}(\x)=p+1$ then we'll call $x_i$ is \textit{maximal elliptic}, and 
		\item If $\ord_{p, i}(\x)=2p$ we say $x_i$ is \textit{maximal parabolic}. 
	\end{enumerate} 
A triple $\x \in \X^*(p)$ will be called \textit{maximal} (hyperbolic, elliptic, or parabolic) if one of its coordinates is either maximal hyperbolic, elliptic, or parabolic.
\end{definition}

\begin{definition} \label{cage} The \textit{cage} is the subgraph $\C(p)$ of the Markoff $\mod p$ graph $\G_p$ containing all vertices that are maximal points in $\X^*(p)$. 
\end{definition}

In \cite{BGS}, the authors prove the following key result. 

\begin{proposition} \label{cageconnected} The cage $\C(p)$ is connected for any prime $p$. \end{proposition}

\begin{remark} \label{why111} Since the action of the Vieta group $\Gamma$ commutes with reduction modulo $p$, there is a correspondence between integer lifts of $\mod p$ solutions to the Markoff equation and paths from $(1, 1, 1)$ to $\x$ in the Markoff $\mod p$ graph $\G_p$, noting that the special point $(1, 1, 1)$ is fixed under reduction modulo $p$. So, if we can guarantee that $(1, 1, 1)$ is connected to the cage $\C(p)$, we can guarantee integer lifts of points in $\C(p)$. In \cite{bellah}, numerical evidence was given to demonstrate that $(1, 1, 1)$ appears to be either in the cage or very close to it, however no explicit results on the connectedness of $(1, 1, 1)$ to $\C(p)$ have appeared in the literature. 
\end{remark}


Our main result, stated below, provides a lower bound on the rotation order of a Markoff $\mod p$ point, under certain conditions, in terms of the $2$-adic valuation of $p+1$. This can then be used to guarantee the connectedness of $(1, 1, 1)$ to $\C(p)$ for primes $p$ where $p+1$ has large enough $2$-adic valuation.\\

Our main result is as follows. 

\begin{theorem} \label{main} 
   Let $p>5$ be prime. If $\x=(x_1, x_2, x_3) \in \X^*(p)$ is a Markoff $\mod p$ point so that for some $i \in \{1, 2, 3\}$ the coordinate $x_i$ is elliptic and
\[\left(\frac{3x_i+2}{p}\right) = -1,\]
then $2^\nu \mid \ord_{p, i}(\x)$, where $\nu=\nu_2(p+1)$.
\end{theorem}

\section{Connection to Pisano Periods} \label{sec:3}

We first consider a key connection to Pisano periods of the Fibonacci sequence. We have the following observation, whose proof is given in Section 3.5 of \cite{bellah}. 

\begin{proposition}
    For all $n \in \Z_{>0}$, we have 
    \[\rot_i^n(1,1,1) = \sigma (1, f_{2n-1}, f_{2n+1}),\] where $f_n$ is the nth value in the Fibonacci Sequence and $\sigma$ is a suitable permutation of the coordinates.
\end{proposition}

Recall that the \textit{Pisano period} $\pi(n)$ of the Fibonacci sequence is the smallest integer $k$ so that
\[F_k \equiv 0 (\mod n) \text{ and } F_{k+1} \equiv 1 (\mod n).\] 
That is, $\pi(n)$ is the length of the Fibonacci sequence modulo $n$. We have the following.

\begin{lemma} \label{Even pisano}
    The Pisano period $\pi(n)$ is even for all $n>2$.
\end{lemma}

We include the proof of this well-known result for completeness. 

\begin{proof}
    Let
    \[\mathcal{F} = \begin{pmatrix}
        1 & 1 \\ 1 & 0
    \end{pmatrix} \in \text{GL}_2(\Z/n\Z),\]
    and observe that for integers $n >2$ we have
    \[\mathcal{F}^n = \begin{pmatrix} f_{n+1} & f_n \\ f_n & f_{n-1} \end{pmatrix},\]
    and so $\mathcal{F}^{\pi(n)}=I$. 
    We have,
    $\det(\mathcal{F}^{\pi(n)}) = \text{det}(\mathcal{F})^{\pi(n)} = (-1)^{\pi(n)}.$
    But from above, we also know 
    $\text{det}(\mathcal{F}^{\pi(n)}) = \text{det}(I) = 1.$
    Thus we must have that 
    \[(-1)^{\pi(n)} = 1.\]
    Therefore, $\pi(n)$ must be even for all integers $n>2$. Note that $\pi(2)=3$ and so this result holds only for $n>2$. 
\end{proof}

\begin{proposition}\label{order_of_(1,1,1)}
    Let $p>5$ be prime and $(1,1,1) \in \X^*(p)$. Then,
    \[\ordpi(1,1,1) = \frac{\pi(p)}{2}.\]
\end{proposition}
\begin{proof} Observe that for all integers $n \geq 1$ we have
\[\rot_1^n(1, 1, 1)=(1, f_{2n-1}, f_{2n+1}),\]
and note that by Lemma \ref{Even pisano} we know that $\pi(p)/2$ is indeed an integer. 
Furthermore, by definition of the Pisano period we have $f_{\pi(p)} \equiv 0\pmod{p}$ and $f_{\pi(p)+1} \equiv 1\pmod{p}$ and so $f_{\pi(p)-1}=f_{\pi(p)+1}+f_{\pi(p)} \equiv 1 (\mod p)$. Hence, $\rot_1^{\pi(n)}(1, 1, 1) =(1, 1, 1)$, and so $\ord_{p,1}(1,1,1) \leq \frac{\pi(p)}{2}$. But since $\pi(p)$ is the smallest positive integer $m$ such that $f_{m-1} \equiv f_{m+1} \equiv 1\pmod{p}$, we must conversely have $\frac{\pi(p)}{2} \leq \text{ord}_{p,1}(1,1,1)$. Note that the proofs for $\text{ord}_{p,2}(1,1,1)$ and $\text{ord}_{p,3}(1,1,1)$ are symmetric to the above. 
\end{proof}

\begin{remark}
    Proposition \ref{order_of_(1,1,1)} tells us that studying rotation orders of $(1, 1, 1)$ is equivalent to studying Pisano periods of the Fibonacci sequence. Since Pisano periods have been extensively studied, this connection is particularly fruitful. In this paper, we demonstrate one consequence of this observation. In particular, we will exploit the following result.
\end{remark}

\begin{theorem}[Theorem 4 of \cite{vince}] \label{vinceResult}
        Let $p = \pm 2\pmod{5}$ and $p+1 = 2^\nu \cdot k$, where $k$ is odd (that is, $\nu=\nu_2(p+1)$). Then $2^{\nu+1} \mid \pi(p)$. 
\end{theorem}

\section{Proof of Main Result and Consequences} \label{sec:4}

Proposition \ref{orders_upperbound} gives upper bounds on the rotation order of a Markoff $\mod p$ point. Our main result (Theorem \ref{main}) generalizes Vince's result (Theorem \ref{vinceResult}) in order to place a a \textit{lower} bound on the rotation order of a Markoff $\mod p$ point under certain conditions in terms of the $2$-adic valuation of $p+1$.\\



The proof Theorem \ref{main} rests in part upon properties of the norm mapping on a quadratic integer ring mod $p$. We will need to use the following well-established properties of the norm.   

\begin{lemma}
    Let $p$ be a prime, and $\Delta \in \Z_p$ be a quadratic non-residue mod $p$. Define the \textit{norm} to be the mapping 
    \begin{align*}
        N: \Z_p[\sqrt{\Delta} ] &\rightarrow \Z_p, \alpha + \beta\sqrt{\Delta} \mapsto  \alpha^2 - \Delta \beta^2.
    \end{align*}
    Then the norm restricted to the multiplicative group of $\Z_p[\sqrt{\Delta}]$
    \[ N:\Big( \Z_p[\sqrt{\Delta}] \Big)^\times \rightarrow \Z_p^\times \]
    is a surjective group homomorphism. 
\end{lemma}

With this tool in hand, we may now prove our main result.

\begin{proof}[Proof of Theorem \ref{main}.] For convenience, set $x=x_i$, $\epsilon=\epsilon_i$ and $\Delta=\Delta_{x_i}$. 
Recall that $\epsilon$ is one of the roots of $f(T)=T^2-3x T+1$. That is, 
\[\epsilon=\frac{3x \pm\sqrt{\Delta}}{2},\]
where $\Delta$ is the discriminant of $f$. Furthermore, recall that $x$ being elliptic gives
$\left(\frac{\Delta}{p}\right)=-1,$
and so $\epsilon$ is in the field $\F_{p^2}$ under the identification 
\[\F_{p^2} \cong \F_p[\sqrt{\Delta}].\]

\smallskip

We first show that $\epsilon$ is a square in $\F_{p^2}$. To this end, denote $a:= \frac{3x}{2}$ and $b = \frac{1}{2}$ (so that $\epsilon = a\pm b\sqrt{\Delta}$).
Observe that
    \begin{align*}
        \text{Tr}(\epsilon) + 2 = \epsilon + \overline{\epsilon} + 2
         :=
        \frac{3x+\sqrt{\Delta}}{2} + \frac{3x-\sqrt{\Delta}}{2} + 2 
        = 2a+2.
    \end{align*}
Thus, our assumption that $3x+2=\Tr(\epsilon)+2$ is quadratic non-residue gives
    \[\left(\frac{(\text{Tr}(\epsilon)+2) \big/ (4\Delta)}{p}\right) 
    = \left(\frac{(\text{Tr}(\epsilon)+2) \big/ \Delta}{p}\right)  = +1\]
by the multiplicativity of the Legendre symbol and the fact that $\Delta$ is a quadratic non-residue. But 
    \[\frac{\text{Tr}(\epsilon)+2}{4\Delta} = \frac{a+1}{2\Delta},\] 
and so we may define $\alpha := \frac{1}{4} \sqrt{\frac{2\Delta}{a+1}}$ and $\beta := \sqrt{\frac{a+1}{2\Delta}}$ which are non-zero elements of $\F_{p}$ by above.
With these definitions, we have 
    \begin{align*}
        \big( \; \alpha+\beta\sqrt{\Delta} \; \big)^2 &= (\alpha^2 + \beta^2\Delta) + (2\alpha\beta)\sqrt{\Delta} 
        = \epsilon
    \end{align*}
as computation confirms that 
    $\alpha^2 + \beta^2\Delta = a$ and $2\alpha\beta = b$. Denoting $\eta := \alpha+\beta\sqrt{\Delta} \in \F_{p^2}$, we thus have $\eta^2=\epsilon$, as needed. 

\smallskip

Next, we show $N(\eta) = -1$. Using the identity $\alpha^2 + \beta^2\Delta = a$ from above, we have 
\begin{align*}
    N(\eta) &=
    N(\alpha +\beta\sqrt{\Delta}) \\
    &= \alpha^2 - \Delta \beta^2 \\
    &= (\alpha^2 + \Delta \beta^2) - 2\Delta \beta^2 \\
    &= a - 2\Delta\left(\frac{a+1}{2\Delta}\right)  \\
    &= -1.
\end{align*}

\smallskip

Next, let $\nu:= \nu_2(p+1)$ denote the $2$-adic valuation of $p+1$; that is, \[p+1 = 2^\nu \cdot k\] where $k$ is odd. Let $K$ denote the elements in $\F_{p^2} = \F_p[\sqrt{\Delta}]$ of norm 1; that is $K$ is the kernel of the restricted norm map $N: \F_{p^2}^\times \to \F_p^\times$. Since $\F_{p^2}$ is a finite field of order $p^2$, we know that $\F_{p^2}^\times$ is a cyclic group of order $p^2-1$, and since $N$ is a surjective homomorphism we know that $K$ is cyclic of order $(p^2-1)/(p-1)=p+1$.
Now let $g$ be a generator of $\F_{p^2}^\times$. Then any element of $K$ must take the form $g^{(p-1)j}$ for some integer $j$. Since $\eta^2$ belongs to $K$ but $\eta$ does not (recalling that we've shown $N(\eta)=-1$), there must be an integer $j$ such that $\eta = g^{(p-1)(j+1/2)}$. Thus the order of $\eta$ in $\F_{p^2}^\times$, which we'll denote by $|\eta|$, equals the smallest positive integer $m$ such that $p^2 - 1 \mid m(p-1)(j+1/2)$, or, equivalently, such that $2(p+1) \mid m (2j+1)$. 
Since $2j+1$ is odd, we obtain $2^{\nu + 1} \; \big| \; |\eta|$, and since $\epsilon=\eta^2$ we have $2^\nu \mid |\epsilon|$. By Lemma \ref{ordepsilon} we know that $|\epsilon|=\ord_{p, i}(\x)$, as needed.
\end{proof}

Theorem \ref{main} implies that for primes $p$ where $p+1$ has large enough $2$-adic valuation, points satisfying the conditions of Theorem \ref{main} are in the cage $\C(p)$ (see Definition \ref{cage}). We have the following.

\begin{corollary} \label{cor1}
    Let $p>5$ be prime and let $\x=(x_1, x_2, x_3) \in \X^*(p)$ be a Markoff $\mod p$ point satisfying the conditions of Theorem \ref{main}. If 
    \[\nu_2(p+1) > \log_2\left(\frac{p+1}{2}\right),\]
    then $\x \in \C(p)$. 
\end{corollary}

\begin{proof} For convenience, set $\nu=\nu_2(p+1)$. Observe that
\begin{align*}
    \nu > \log_2\left(\frac{p+1}{2}\right) \quad\Rightarrow\quad  \frac{p+1}{2} < 2^\nu \;\leq \; \ord_{p, i}(\x), 
\end{align*}
where the final inequality follows by Theorem \ref{main}. So, $\ord_{p, i}(\x) = p+1$ which gives $\x \in \C(p)$ by Lemma \ref{orders_upperbound}. \end{proof}

In particular, we have the following. 

\begin{corollary} \label{cor2} If $p \equiv \pm 2 (\mod 5)$ is a Mersenne prime, then $(1, 1, 1) \in \C(p)$. \end{corollary}

\begin{proof} 
Observe that $\epsilon_1=(3+\sqrt{5})/2$
and so $\Delta_1=\Tr(\epsilon_1)+2 = 5$. The assumption that $p \equiv \pm 2\pmod{5}$ implies that $\Delta_1$ and $\text{Tr}(\epsilon_1)+2$ are both quadratic non-residues mod $p$, by quadratic reciprocity. So the conditions of Theorem \ref{main} hold for $(1, 1, 1)$. Furthermore, since $p$ is Mersenne we can write $p=2^n-1$ and so 
\[\nu_2(p+1)=n > n-1 = \log\left(\frac{p+1}{2}\right),\]
which gives $(1, 1, 1) \in \C(p)$ by Theorem \ref{main}. 
\end{proof}

\begin{remark} As discussed in Remark \ref{why111}, if $(1, 1, 1)$ is connected to the cage, then all points $\x$ in $\C(p)$ have integer lifts. Thus, Theorem \ref{mainprelim} follows as a consequence of Corollary \ref{cor2}. 
\end{remark}


There are several ways in which future research could use or expand upon Theorem \ref{main}. For example, it may be fruitful to explore any probabilistic conclusions that can be drawn from the result, such as a lower bound on the 
percentage of points in $\mathcal{G}_p$ that lie in the cage for primes $p$ satisfying the conditions of the theorem. 
Additionally, further work will be needed to understand how the proof of this theorem might be generalized, for example by considering other divisors of $p+1$.



\section{Density of Primes with \texorpdfstring{$(1, 1, 1)$}{(1, 1, 1)} in the Cage} \label{sec:5}

We conclude by investigating how often $(1, 1, 1)$ is in the cage. Define
    \begin{equation}\label{density}
    \delta := \lim_{x\to \infty}\frac{ \# \{\text{primes } p\leq x : (1,1,1) \in \mathcal{C}(p) \}}{\# \{\text{primes } p\leq x \}}.
    \end{equation}

 The following Proposition will yield an upper bound on $\delta$.

\begin{proposition}\label{hyp_cage}
            If $(1,1,1)$ is hyperbolic, then $(1,1,1)$ is not in $\C(p)$.
        \end{proposition}
        
        \begin{proof}
            Let $p$ be a prime such that the value $1$ is hyperbolic. Then, by definition
            $\Delta_1 = (3\cdot 1)^2 - 4 = 5$
            is a non-zero square mod $p$. That is, $\sqrt{5} \in \mathbb{F}_p$ and so
            $\epsilon_1 = \frac{3 + \sqrt{5}}{2} \in \F_p. $
            Observe that $\epsilon_1$ is a nonzero square in $\F_p$, since 
            \[\frac{3 + \sqrt{5}}{2} = \left(\frac{1+\sqrt{5}}{2}\right)^2,\]
            and so $\epsilon$ has order at most $\frac{|\F_p^\times|}{2} < p-1$.
            Thus, $1$ is not maximal by Lemma \ref{orders_upperbound} and so $(1,1,1)$ is not in the cage $\C(p)$. 
        \end{proof}

\begin{theorem}
With $\delta$ defined as in Equation (\ref{density}) we have $\delta<0.5$. \end{theorem}

\begin{proof} By Proposition \ref{hyp_cage}, if $(1, 1, 1) \in \C(p)$ then the value 1 is not hpyerbolic. That is, $\Delta_1 = 5$ is a quadratic non-residue mod $p$, and so by quadratic reciprocity, $p$ is a quadratic non-residue mod $5$. Hence, $p \equiv 2, 3 \,(\mod 5)$ which gives
    \begin{align*}
        \delta &= \lim_{x\to \infty}\frac{ \# \{\text{primes } p\leq x : (1,1,1) \in \mathcal{C}(p) \}}{\# \{\text{primes } p\leq x \}} \\
        & < \lim_{x\to \infty}\frac{ \# \{\text{primes } p \equiv 2, 3 \,(\mod 5)\}}{\# \{\text{primes } p\leq x \}}
    \end{align*}
    and so $\delta<0.5$ by Dirichlet's Theorem on primes in arithmetic progressions. 
\end{proof}

Through numerical experimentation, we conjecture that $\delta$ exists and approaches roughly 40\%, as demonstrated in Figure 1 below. 
\begin{figure}[H] 
    \centering
    \includegraphics[width=0.9\linewidth]{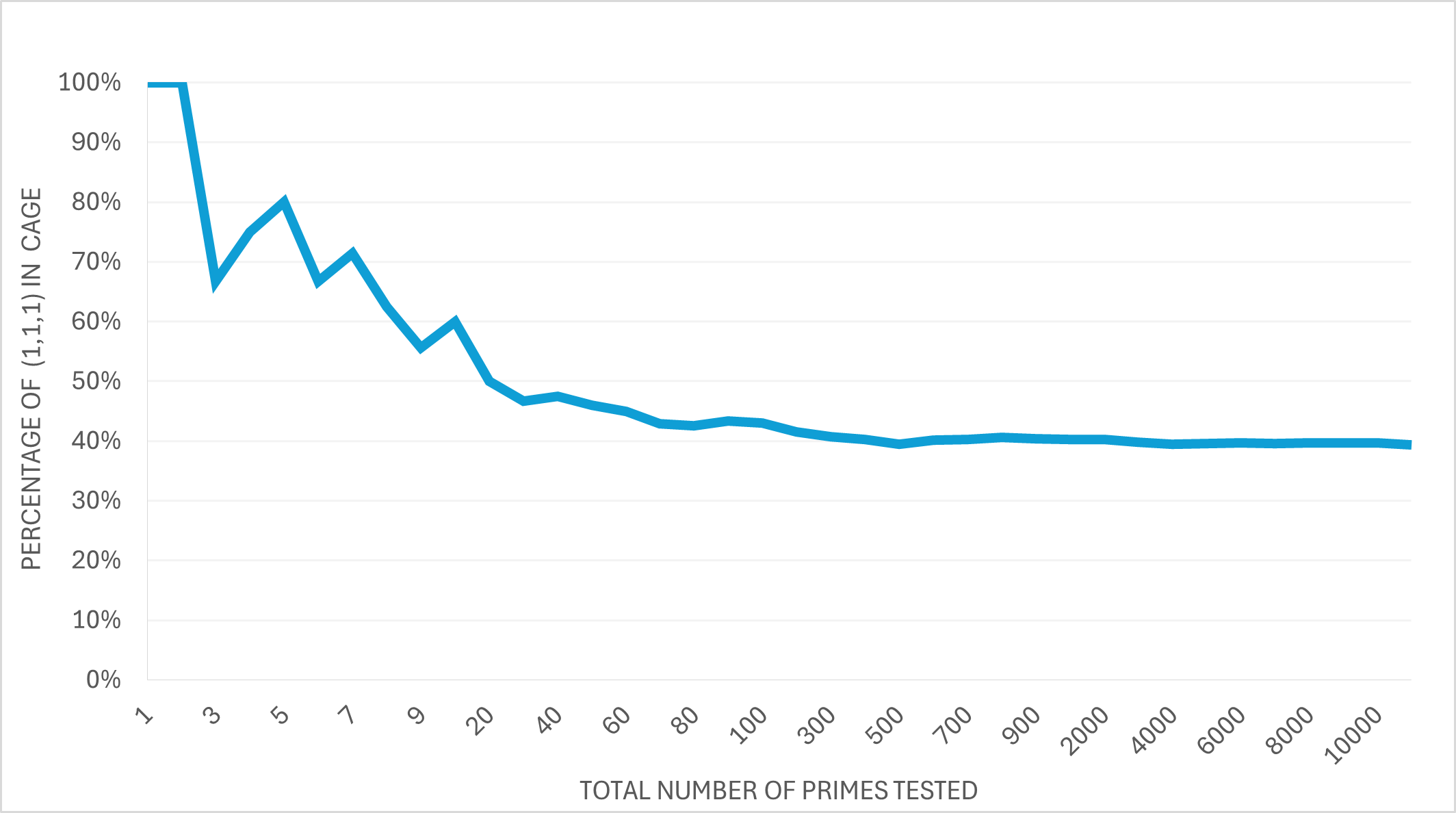}
    \caption{Percent of primes $p$ with $(1,1,1) \in \C(p)$}
    \label{fig:1}
\end{figure}
    
An interesting next direction would be to use the results of Theorem \ref{main} to obtain a lower bound on $\delta$.


\bibliography{bib}
\bibliographystyle{plain}

\end{document}